\documentclass[a4paper, 11pt]{article}

\usepackage[arxiv]{optional}
\usepackage[english]{babel}
\usepackage[T1]{fontenc}
\usepackage{amsmath}
\opt{arxiv}{%
\usepackage{amsthm}
}%
\usepackage{amsfonts}
\usepackage{amssymb}
\usepackage[final]{graphicx}
\usepackage{paralist}
\usepackage{hyperref}
\usepackage{tikz}

\usepackage{cleveref}
\usepackage{subfigure}
\usepackage{pifont}
\usepackage{caption}
\usepackage{mathtools}

\opt{draft}{%
\usepackage{refcheck}

}

\opt{arxiv}{%
\theoremstyle{plain}
\newtheorem{thm}{Theorem}[section]
\newtheorem{prop}[thm]{Proposition}
\newtheorem*{prop*}{Proposition}
\newtheorem{lem}[thm]{Lemma}
\newtheorem*{lem*}{Lemma}
\newtheorem{cor}[thm]{Corollary}
\newtheorem*{cor*}{Corollary}

\newtheorem*{example*}{Example}
\newtheorem{conject}[thm]{Conjecture}
\newtheorem*{conject*}{Conjecture}

\theoremstyle{definition}
\newtheorem{defn}[thm]{Definition}
\newtheorem*{defn*}{Definition}
\newtheorem{rem}[thm]{Remark}
\newtheorem*{rem*}{Remark}
}%
\opt{springer}{%

}%

\newcommand{\eps}{\varepsilon}
\renewcommand{\emptyset}{\varnothing}

\opt{springer}{%
\newcommand{\eop}{\qed}
}
\opt{arxiv}{%
\newcommand{\eop}{}
}

\newcommand{\C}{\ensuremath{\mathbb{C}}}

\newcommand{\N}{\ensuremath{\mathbb{N}}}

\newcommand{\defby}{\mathrel{\mathop:}=}

\newcommand{\conj}[1]{\overline{#1}}

\DeclareMathOperator{\intx}{int}

\title{The maximum number of zeros of $r(z) - \conj{z}$ revisited}

\author{J\"org Liesen\footnotemark[1] \and Jan Zur\footnotemark[1]}
\date{December 7, 2017}

\begin{document}

\maketitle

\renewcommand{\thefootnote}{\fnsymbol{footnote}}

\footnotetext[1]{TU Berlin, Institute of Mathematics, MA 4-5, Stra{\ss}e des 17. Juni 136, 10623 Berlin, Germany.
\texttt{\{liesen,zur\}@math.tu-berlin.de}}

\renewcommand{\thefootnote}{\arabic{footnote}}

\begin{abstract}
Generalizing several previous results in the literature on rational harmonic
functions, we derive bounds on the maximum number of zeros of functions
$f(z) = \frac{p(z)}{q(z)} - \conj{z}$, which depend on both $\deg(p)$ and 
$\deg(q)$. Furthermore, we prove that any function that attains one of these
upper bounds is regular. 
\end{abstract}

\section{Introduction}\label{sect:introduction}

We study the zeros of \emph{rational harmonic functions} of the form
\begin{align}\label{eq:f}
f(z) = r(z)- \conj{z}\quad\mbox{with}\quad r(z)=\frac{p(z)}{q(z)},
\end{align}
where $p$ and $q$ are coprime polynomials of respective degrees $n_p$ and $n_q$, and
$$n := {\rm deg}(r)=\max\{n_p,n_q\}\geq 2.$$
If $z_0\in\C$ is a zero of $f$, i.e., $r(z_0)-\conj{z}_0=0$, then also $z_0=\conj{r(z_0)}$.
Inserting this into the first equation and taking complex conjugates gives
$\conj{r}(r(z_0))-z_0=0$. This can be transformed into a polynomial equation (of degree $n^2+1$),
which shows that $f$ has finitely many zeros. It is also important to note that because
of the term $\conj{z}$, the zeros of $f$ can not be ``factored out'', and hence they do
not have a multiplicity in the usual sense. ``Numbers of zeros'' in this context
therefore refer to numbers of distinct complex points.

Several authors have studied upper bounds on $N(f)$, the number of zeros
of a function $f$ as in \eqref{eq:f}. In particular, Khavinson and Neumann~\cite{KhavinsonNeumann2006}
showed that in general $N(f)\leq 5n-5$, and Khavinson and {\'S}wi{\c{a}}tek~\cite{KhavinsonSwiatek2003} showed that if $n_q=0$, i.e., $f$ is a harmonic polynomial, then $N(f)\leq 3n-2$.
Results of Rhie~\cite{Rhie2003} and Geyer~\cite{Geyer2008}, respectively, show that these two upper bounds are sharp for each $n\geq 2$.

The main purpose of this note is to prove the following theorem, which takes the
individual degrees $n_p$ and $n_q$ into account, and which generalizes (almost)
all previously known bounds on the maximal number of zeros of~$f$.

\begin{thm}\label{thm:main}
Let $f$ be as in \eqref{eq:f}. Then for every $c\in\C$, the number of zeros of
$f_c(z):=f(z)-c$ satisfies
\begin{align*}
N(f_c) \le \begin{cases}
2n_p + 3n_q - 3, \quad &\text{ if } n_p < n_q, \\
5n_p-5,\quad &\text{ if } n_p=n_q,\\
3n_p + 2n_q - 2, \quad &\text{ if } n_p > n_q + 1.
\end{cases}
\end{align*}
\end{thm}

The only case ``missing'' in this theorem is $n_p=n_q+1$. For this case we know
from~\cite{LuceSeteLiesen2014a} that $N(f_c)\leq 5n_p-6$.

Note that for each $c\in\C$ we can write
\begin{equation}\label{eq:fc}
f_c(z)=r_c(z)-\conj{z},\quad\mbox{where}\quad
r_c(z):=\frac{p(z)-cq(z)}{q(z)},
\end{equation}
which is again a rational harmonic function of the form \eqref{eq:f}. The degree of
the numerator polynomial of $r_c$ is potentially different from $n_p$, and this allows
for some flexibility in applications of Theorem~\ref{thm:main}.
A second reason why we have formulated the result for the function $f_c$
(rather than just $f$) is the application of rational harmonic functions in the context of gravitational
lensing; see~\cite{KhavinsonNeumann2008} for a survey. In that application a constant
shift represents the position of the light source of the lens, and the change of the
number of zeros under movements of the light source is of great interest;
see, e.g.,~\cite{LiesenZur2017} for more details.

Our note is organized as follows. In Section~2 we briefly recall the mathematical
background, in particular the argument principle for continuous functions and
a helpful result from complex dynamics. In Section~3 we prove Theorem~\ref{thm:main}.
We also prove that a function that attains the bound in Theorem~\ref{thm:main} is
regular, which generalizes a result from~\cite{LuceSeteLiesen2014b}. In Section~4
we explain why the special case $n_p=n_q+1$ cannot be completely
resolved by our method of proof, and we discuss the relation of Theorem~\ref{thm:main} 
to all previously published bounds that we are aware of.

\section{Mathematical background}

Let $f$ be as in \eqref{eq:f}. Using the Wirtinger derivatives $\partial_z$ and $\partial_{\bar{z}}$
we can write the \emph{Jacobian} of $f$ as
$$J_f(z) = |\partial_z f(z)|^2-|\partial_{\bar{z}} f(z)|^2=|r'(z)|^2-1.$$
If $z_0\in\C$ is a zero of $f$, i.e., $f(z_0)=0$, then $z_0$ is called a \emph{sense-preserving},
\emph{sense-reversing}, or \emph{singular zero} of $f$, if $|r'(z_0)| > 1$, $|r'(z_0)| < 1$,
or $|r'(z_0)| = 1$, respectively. The sense-preserving and sense-reversing zeros of $f$ are
called the \emph{regular zeros}. If $f$ has only such zeros, then $f$ is called \emph{regular},
and otherwise $f$ is called \emph{singular}. 
We denote the number of sense-preserving, sense-reversing, and singular zeros of $f$ in
a set $S\subseteq\C$ by $N_{+}(f;S)$, $N_{-}(f;S)$, and $N_{0}(f;S)$,
respectively. For $S=\C$ we simply write $N_{+}(f)$, $N_{-}(f)$, and $N_{0}(f)$.
In our proofs we will use the following result on 
regular functions; see~\cite[Lemma]{KhavinsonNeumann2006}.

\begin{lem}\label{lem:density}
If $f$ is as in \eqref{eq:f}, then the set of complex numbers $c$ for which $f_c(z) = f(z) - c$
is regular, is open and dense in $\C$.
\end{lem}

This lemma can be easily shown when using the fact that the function $f_c$ is singular
if and only if $c$ is a \emph{caustic point} of $f$; 
see~\cite[Proposition~2.2]{LiesenZur2017}.

Let $\Gamma$ be a closed Jordan curve, and let $g$ be any function that is continuous and nonzero on $\Gamma$.
Then the \emph{winding} of $g$ on~$\Gamma$ is defined as the change in the argument of $g(z)$ as $z$ travels
once around $\Gamma$ in the positive direction, divided by $2\pi$, i.e.,
$$V(g;\Gamma) \defby \frac{1}{2\pi}\Delta_\Gamma \arg g(z).$$
The following result holds for the winding of the functions of our interest.

\begin{thm}\label{thm:zero_counting}
Let $f$ be as in \eqref{eq:f}. If $f$ is nonzero and finite on a closed Jordan curve $\Gamma$
and has no singular zero in $\intx(\Gamma)$, then
\begin{align*}
V(f;\Gamma) = N_+(f;\intx(\Gamma)) - N_-(f;\intx(\Gamma)) - P(f;\intx(\Gamma)),
\end{align*}
where $P(f;\intx(\Gamma))$ denotes the number of poles with multiplicities of $r$, 
and hence of $f$, in $\intx(\Gamma)$.
\end{thm}

We will frequently use the following version of Rouch{\'e}'s theorem.

\begin{thm}\label{thm:Rouche}
Let $\Gamma$ be a closed Jordan curve and suppose that $f,g : \Gamma \rightarrow \C$ are continuous.
If $|f(z) - g(z)| < |f(z)|+|g(z)|$ holds for all $z \in \Gamma$, then $V(f;\Gamma) = V(g;\Gamma)$.
\end{thm}

For more details on the mathematical background described above we refer to
\cite{KhavinsonNeumann2006}, \cite{LiesenZur2017} and \cite{SeteLuceLiesen2015a}.

In addition, we will need a result on fixed points from complex dynamics. Let $z_f \in \C$ be a fixed point
of a rational function $r$, i.e., $r(z_f) = z_f$.  Then $z_f$ is called \emph{attracting}, \emph{repelling},
or \emph{rationally neutral}, if respectively $|r'(z_f)| < 1$, $|r'(z_f)| > 1$, or $|r'(z_f)| = 1$. 
The following is a combination of \cite[Chapter~III, Theorem~2.2 and~2.3]{CarlesonGamelin1993}.
\begin{thm}\label{thm:fixed_point}
If $z_f$ is an attracting or rationally neutral fixed point of a rational function $R$ 
with $\deg(R) \ge 2$, then exists a critical point $z_c$ of $R$, i.e., $R'(z_c) = 0$, 
with $\lim_{k\rightarrow \infty} R^k(z_c) = z_f$, where $R^k:=R\circ\cdots\circ R$ ($k$ times).
\end{thm}

\section{Main results}

Our strategy to prove Theorem~\ref{thm:main} is the following: First we determine $N_+(f_c) - N_{-}(f_c)$
for regular functions $f$ as in \eqref{eq:f} and $f_c(z)=f(z)-c$ with respect to $n_p$ and $n_q$ using
Theorem~\ref{thm:zero_counting} and~\ref{thm:Rouche}. Then we bound $N_{0}(f)+N_{-}(f)$ for general~$f$
by the number of zeros of $r'$ using Theorem~\ref{thm:fixed_point}. Finally, we combine both results
with Lemma~\ref{lem:density} in order to obtain the proof also for non-regular $f$.

We denote by $B_M(z_0)$ the open disk of radius $M>0$ around $z_0\in\C$.

\begin{lem}\label{lem:N+-N-}
Let $f$ be as in \eqref{eq:f} and suppose that $f$ is regular, i.e., $N_{0}(f)=0$.
Then for every $c\in\C$ the function $f_c(z)=f(z)-c$ satisfies
\begin{align*}
N_+(f_c) - N_-(f_c) =
\begin{cases}
n_q-1, \quad &\text{ if } n_p \le n_q, \\
n_p, \quad &\text{ if } n_p > n_q + 1.
\end{cases}
\end{align*}
\end{lem}

\begin{proof}
For each $c\in\C$ we can write $f_c$ as in \eqref{eq:fc}. Obviously, $r$ and $r_c$
have the same poles, and the argument given in the Introduction shows that $f_c$
has finitely many zeros. Thus, if $\widetilde{M}>0$ is sufficiently large,
we have $P(f_c;B_{\widetilde{M}}(0))=n_q$ and
$N_{+/-}(f_c;B_{\widetilde{M}}(0))=N_{+/-}(f_c)$.

First we assume $n_p\le n_q$. Then $\lim_{|z|\rightarrow \infty}|r(z)|$
is finite (possibly zero), and for a sufficiently large $M\geq \widetilde{M}>0$,
$$|f_c(z) + \conj{z}| = |r(z)-c| < |z| \le |f_c(z)| + |\conj{z}|
\quad \text{ for all } z\in\partial B_{M}(0).$$
Using Theorem~\ref{thm:zero_counting} and~\ref{thm:Rouche},
\begin{align*}
-1 &= V(-\conj{z};\partial B_M(0)) = V(f_c;\partial B_M(0)) \\
&= N_+(f_c;B_M(0)) - N_-(f_c;B_M(0)) - P(f_c;B_M(0)) \\
&= N_+(f_c) - N_-(f_c) - n_q,
\end{align*}
which gives $N_+(f_c) - N_-(f_c)=n_q-1$.

Next, we assume $ n_p > n_q+1$. Then $r$ can be written as $r(z) = \widetilde{p}(z) + \widetilde{r}(z)$,
where $\widetilde{p}$ is a polynomial of exact degree $n_p-n_q \ge 2$, and $\widetilde{r}$ is a rational
function with $\lim_{|z| \rightarrow \infty}|\widetilde{r}(z)| = 0$. For a sufficiently
large $M\geq\widetilde{M}>0$,
$$|f_c(z) - \widetilde{p}(z)| = | \widetilde{r}(z) - \conj{z} - c| < |\widetilde{p}(z)| \le
|f_c(z)| + |\widetilde{p}(z)|
\quad \text{ for all } z\in\partial B_{M}(0).$$
Using again Theorem~\ref{thm:zero_counting} and~\ref{thm:Rouche},
\begin{align*}
n_p-n_q &= V(\widetilde{p};\partial B_M(0)) = V(f_c;\partial B_M(0)) \\
&= N_+(f_c;B_M(0)) - N_-(f_c;B_M(0)) - P(f_c;B_M(0)) \\
&= N_+(f_c) - N_-(f_c) - n_q,
\end{align*}
and hence $n_p = N_+(f_c) - N_-(f_c)$.\eop
\end{proof}

Much of the proof of the next result is based on the proof of~\cite[Theorem C.3]{AnEvans2006}.
We nevertheless include all steps for clarity and completeness of our presentation.

\begin{lem}\label{lem:N0+N-}
Let $f$ be as in \eqref{eq:f}. Then for every $c\in\C$ the function $f_c(z)=f(z)-c$ satisfies
$$N_0(f_c)+N_{-}(f_c)\le n_p+n_q-1.$$
\end{lem}

\begin{proof}
Let $z_0\in\C$ be a non-sense-preserving zero of $f$, i.e., $r(z_0) = \conj{z}_0$ with
$|r'(z_0)| \le 1$. Then also $\conj{r}(r(z_0))=z_0$, or
$$R(z_0)=z_0,\quad\mbox{where}\quad R(z):=(\conj{r}\circ r)(z).$$
The derivative of the rational function $R$ is given by
$$ R'(z) = (\conj{r} \circ r )'(z) = \conj{r}'(r(z))r'(z).$$
Thus,
$$ |R'(z_0)| = |\conj{r}'(r(z_0))r'(z_0)| = |\conj{r}'(\conj{z}_0)r'(z_0)| = |\conj{r'(z_0)}r'(z_0)| =
|r'(z_0)|^2 \le 1,$$
which shows that $z_0$ is an attracting or rationally neutral fixed point of $R$. By Theorem~\ref{thm:fixed_point},
there exits a critical point $z_c$ of $R$, i.e.,
\begin{equation}\label{eqn:Rzc}
 R'(z_c) = \conj{r}'(r(z_c))r'(z_c) = 0,
\end{equation}
such that
\begin{equation}\label{eqn:limR}
\lim_{k \rightarrow \infty} R^k(z_c) = z_0.
\end{equation}

From \eqref{eqn:Rzc} we obtain $r'(z_c)=0$ or $r'(\conj{w}_c) = 0$, where $w_c := r(z_c)$.
In the second case we use \eqref{eqn:limR} and the continuity of $\conj{r}$ to obtain
\begin{align*}
\lim_{k \rightarrow \infty}R^k(\conj{w}_c)
&= \lim_{k \rightarrow \infty}(\conj{r} \circ r)^k(\conj{w}_c)
=\lim_{k \rightarrow \infty}(\conj{r} \circ r)^k(\conj{r}(\conj{z}_c)) \\
&= \conj{r}\Big(\lim_{k\rightarrow \infty} (r\circ \conj{r})^k(\conj{z}_c)\Big)
= \conj{r}\Big(\conj{\lim_{k\rightarrow \infty}( \conj{r}\circ r)^k(z_c)}\Big) \\
&= \conj{r}\Big(\conj{\lim_{k \rightarrow \infty}R^k(z_c)}\Big) = \conj{r}(\conj{z}_0) = \conj{r(z_0)} = z_0.
\end{align*}

In summary, we have shown that if $z_0$ is a non-sense-preserving zero of $f$, and hence
an attracting or rationally neutral fixed point of $R$, then there exists a critical point
$\widetilde{z}_c$ of $r$ (in the first case $\widetilde{z}_c=z_c$, and in the second case
$\widetilde{z}_c=\conj{w}_c$) with $\lim_{k \rightarrow \infty} R^k(\widetilde{z}_c) = z_0$.
Clearly, different fixed points of $R$ attract disjoint sets of (critical) points,
and therefore $N_{0}(f)+N_{-}(f)$ is less than or equal to the number of zeros of
$$r'(z)=\frac{p'(z)q(z)-p(z)q'(z)}{q(z)^2},$$
which is at most $n_p+n_q-1$. Finally, for every $c\in\C$ we can write $f_c$ in the
form \eqref{eq:fc}. Now, by the same argument as above, $N_{0}(f_c)+N_{-}(f_c)$ is
less than or equal to the number of zeros of $r_c'=r'$, which is at most $n_p+n_q-1$.
\eop
\end{proof}

In order to control the behavior of singular functions we will also need the fact
that a small constant perturbation does not reduce the number of sense-preserving
zeros of $f$.

\begin{lem}\label{lem:N+N+}
Let $f$ be as in \eqref{eq:f} and let $z_1,\dots,z_k$ be the sense-preserving zeros of $f$.
Then $N_+(f) \le N_+(f_c)$ for all $c\in\C$ with sufficiently small $|c|$.
\end{lem}
\begin{proof}
Since $f$ has finitely many zeros we can always find an $\eps>0$,
such that $f$ is sense-preserving on $B_\eps(z_j)$ for $j = 1, \dots, k$,
and $B_\eps(z_j) \cap B_\eps(z_\ell) = \emptyset$ for $j \neq \ell$, and
$$\delta := \min\big\{|f(z)| : z \in \cup_{j=1}^k\, \partial B_\eps(z_j)\big\}>0.$$
In particular, the condition $\delta>0$ just means that none of the boundaries
$\partial B_\eps(z_j)$ contains a zero of $f$. By construction, for all
$z \in \cup_{j=1}^k\,\partial B_\eps(z_j)$ we have for $|c| < \delta$
$$|f(z) - f_c(z)| = |c| < \delta \le |f(z)| \le  |f(z)| + |f_{c}(z)|.$$
With Theorem~\ref{thm:Rouche} we get, for each $j=1,\dots,k$,
$$ 1 = N_+(f;B_\eps(z_j)) = V(f;\partial B_\eps(z_j)) = V(f_c;\partial B_\eps(z_j)) = N_+(f_c;B_\eps(z_j)),$$
where in the third equality we used that a constant shift preserves the orientation of $f$ on $\C$. \eop
\end{proof}
\paragraph{Proof of Theorem~\ref{thm:main}:}
Let $f$ be as in \eqref{eq:f} and let $c\in\C$ be arbitrary. Due to Lemma~\ref{lem:density},
there exists a sequence $\{c_k\}_{k\in \N} \subset \C$, such that the functions
$f_{c_k}(z) := f(z) - c_k$ are regular, and $c_k \rightarrow c$. If $f_c$ is regular,
we can chose $c_k= c$ for all $k\in\N$.

If $n_p<n_q$, then for sufficiently small $|c_k-c|$,
\begin{align*}
N(f_c) &= N_+(f_c) + N_{0}(f_c) + N_{-}(f_c) \\
&\le N_{+}(f_{c_k}) - N_{-}(f_{c_k}) + N_{-}(f_{c_k}) + N_{0}(f_c) + N_{-}(f_c) \\
& \le (n_q - 1) + (n_p + n_q - 1) + (n_p + n_q - 1) \\
&= 2n_p + 3n_q - 3,
\end{align*}
where we have used Lemma~\ref{lem:N+-N-}--\ref{lem:N+N+}. Analogously, if $n_p > n_q+1$, then
\begin{align*}
N(f_c) &\le N_{+}(f_{c_k}) - N_{-}(f_{c_k}) + N_{-}(f_{c_k}) + N_{0}(f_c) + N_{-}(f_c) \\
& \le n_p + (n_p + n_q - 1) + (n_p + n_q - 1)  \\
& =  3n_p + 2n_q - 2.
\end{align*}
Finally, if $n_p=n_q$, then the rational function $r$ in \eqref{eq:f} can be written as
$$r(z)=\frac{\alpha z^n+\widetilde{p}(z)}{z^n+\widetilde{q}(z)}=
\frac{\widetilde{p}(z)-\alpha\widetilde{q}(z)}{z^n+\widetilde{q}(z)}+\alpha,$$
where $\alpha\neq 0$, ${\rm deg}(\widetilde{p})<n_p=n$ and ${\rm deg}(\widetilde{q})<n_q=n$.
Now the numerator degree is
$$\widetilde{n}_p:={\rm deg}(\widetilde{p}-\alpha\widetilde{q})\le n_p-1,$$
and applying the bound from the first case to the function
$$f_{-\alpha}(z):=
\frac{\widetilde{p}(z)-\alpha\widetilde{q}(z)}{z^n+\widetilde{q}(z)}-\conj{z}+\alpha$$
gives
$$N(f)=N(f_{-\alpha})\leq 2\widetilde{n}_p+3n_q-3\le 5n_p-5,$$
which completes the proof.

\bigskip

We will now show that any function $f$ as in \eqref{eq:f} that attains one of the bounds of
Theorem~\ref{thm:main} is regular, which generalizes~\cite[Theorem~3.1]{LuceSeteLiesen2014b}.

\begin{lem}\label{lem:N+}
Let $f$ be as in \eqref{eq:f} and suppose that $f$ is singular. Then there exists a constant $c\in\C$
such that $N_+(f) < N_+(f_{c})$.
\end{lem}

\begin{proof}
Note that $f$ has at least one sense-preserving zero due to Lemma \ref{lem:N+-N-}.
Let $z_1, \dots, z_k$ be the sense-preserving zeros of $f$ with the corresponding disks
$B_\eps(z_1),\dots,B_\eps(z_k)$ as well as $\delta>0$ as in the proof of Lemma~\ref{lem:N+N+}. 

Let $z_0$ be a singular zero of $f$.
We then  have $|f(z)| < \delta$ in $B_{\widetilde{\eps}}(z_0)$ and
$$B_{\widetilde{\eps}}(z_0)\cap B_\eps(z_j)=\emptyset,\quad j=1,\dots,k,$$
if $\widetilde{\eps} > 0$ is small enough. Let $\widetilde{z} \in B_{\widetilde{\eps}}(z_0)$
be arbitrary with $|r'(\widetilde{z})| > 1$. Then $|f(\widetilde{z})|<\delta$, and the
proof of Lemma~\ref{lem:N+N+} shows that the function $\widetilde{f} := f - f(\widetilde{z})$
has one sense-preserving zero in each of the $k$ disks $B_\eps(z_1),\dots,B_\eps(z_k)$.
Moreover, $\widetilde{f}$ has an additional sense-preserving zero at $\widetilde{z}$,
which means that $N_+(f) < N_+(\widetilde{f})$.\eop
\end{proof}

Note that the bounds of Lemma~\ref{lem:N+N+} and \ref{lem:N+} also hold for $N_-$.

\begin{thm} 
Let $f$ as in \eqref{eq:f} and $c\in\C$. If $f_c$ attains one of the bounds of Theorem~\ref{thm:main}, then $f_c$ is regular.
\end{thm}

\begin{proof}
Let $f$ be as in \eqref{eq:f} with $n_p \le n_q$, and let $c\in\C$ be arbitrary.
Due to Lemma~\ref{lem:density} we can choose a sequence $\{c_k\}_{k\in\N}$, such
that the functions $f_{c_k}$ are regular and $c_k \rightarrow c$. Using Lemma~\ref{lem:N+-N-}--\ref{lem:N+N+} we obtain
\begin{align}
\begin{aligned}\label{eq:tmp_1}
N_+(f_c) &\le N_+(f_{c_k}) = N_+(f_{c_k}) - N_-(f_{c_k}) + N_{-}(f_{c_k})\\
&\le (n_q - 1) + (n_p + n_q - 1) = n_p + 2n_q - 2.
\end{aligned}
\end{align}
Now suppose that $f_c$ attains the bound of Theorem~\ref{thm:main}. Then by
Lemma~\ref{lem:N0+N-} we get
\begin{align*}
2n_p+3n_q-3 &= N(f_c)=N_{+}(f_c)+N_{-}(f_c)+N_0(f_c) \\
&\le N_{+}(f_{c_k})+(n_p+n_q-1),
\end{align*}
and with  \eqref{eq:tmp_1} we obtain $N_{+}(f_c)=n_p+2n_p-2$. If $f_c$ would be
singular, we could choose a constant $\widetilde{c}\in\C$ such
that $N_{+}(f_{\widetilde{c}})>N_{+}(f_c)$, but
this is in contradiction to the upper bound \eqref{eq:tmp_1}, which holds
for an arbitrary constant.

The proof for the case $n_p > n_q + 1$ is analogous. \eop
\end{proof}

\section{Discussion of Theorem~\ref{thm:main}}

The reason why the special case $n_p=n_q+1$ is ``missing'' in Theorem~\ref{thm:main}
is that this case is not covered in Lemma~\ref{lem:N+-N-}. Note that in this case we can write
$r(z)=\alpha z+\widetilde{r}(z)$ for some $\alpha\neq 0$ and with 
$\lim_{|z|\rightarrow \infty} |\widetilde{r}(z)|$ finite (possibly zero). 
If $|\alpha|>1$, then for a sufficiently large $M\geq \widetilde{M}>0$ we obtain
(cf. the proof of Lemma~\ref{lem:N+-N-})
\begin{align*}
|f_c(z)-\alpha z| &=|\widetilde{r}(z)-\conj{z}-c|<|\alpha z| \leq 
|f_c(z)|+|\alpha z| 
\end{align*}
for all $z\in\partial B_{M}(0)$, leading to $1=n_p-n_q=N_+(f_c) - N_-(f_c)-n_q$, or 
$N_+(f_c) - N_-(f_c)=n_q+1=n_p$. This gives the bound $N(f_c)\leq 3n_p+2n_q-2=5n_p-4$, 
but we know that $N(f_c)\leq 5n_p-6$ from~\cite{LuceSeteLiesen2014a}. For the case
$|\alpha|=1$ the method of proof used for Lemma~\ref{lem:N+-N-} would give no result, 
and for $|\alpha|<1$ we would indeed obtain $N(f_c)\leq 5n_p-6$, since then
$$|f_c(z)+\conj{z}|=|\widetilde{r}(z)+\alpha z-c|<|z|\leq |f_c(z)|+|\conj{z}|$$
for all $z\in\partial B_{M}(0)$, giving $N_+(f_c) - N_-(f_c)=n_q-1=n_p-2$.

Apart from the special case $n_p=n_q+1$, Theorem~\ref{thm:main} covers all 
possible choices of $n_p$ and $n_p$ with $n=\max\{n_p,n_q\}\geq 2$, and all previous results 
in this area that we are aware of. In particular:
\begin{enumerate}
\item[(i)] For any choices of $n_p$ and $n_p$ in Theorem~\ref{thm:main} (except $n_p=n_p+1$) 
we get $N(f_c)\leq 5n-5$, which is the general bound from~\cite{KhavinsonNeumann2006}.
\item[(i)] For $n=n_p\geq 2$ and $n_q=0$, Theorem~\ref{thm:main} gives the bound for
harmonic polynomials from~\cite{KhavinsonSwiatek2003}.
\item[(iii)] For $n=n_p>n_q+1$, Theorem~\ref{thm:main} gives the same
bound as in~\cite{LeeMakarov2014}.
\item[(iv)] For $n=n_p=n_q+j$ with $j\geq 2$ we get $N(f)\le 5n-2(j+1)$. For $j=2$ this
is the same bound as in~\cite{LuceSeteLiesen2014b}, and for $j>2$ our new bound is
smaller than the bound in~\cite{LuceSeteLiesen2014b} ($5n-6$).
\item[(v)] For $n=n_q=n_p+j$ with $j\geq 1$ we get $N(f)\le 5n-(2j+3)$.
For $j=1$ this is the same bound as in~\cite{KhavinsonNeumann2006}, and for
$j>1$ our new bound is smaller than the previous one.
\end{enumerate}

The following upper bounds on the maximal number of zeros of $f$ as in \eqref{eq:f}
have been shown to be sharp:

\begin{align*}
N(f) \le \begin{cases}
3n-2, \quad &\text{ if } (n_p,n_q)=(n,0)~\cite{Geyer2008}, \\
5n-5,\quad &\text{ if } (n_p,n_q)=(n-1,n)~\cite{Rhie2003},\\
5n-6, \quad &\text{ if } (n_p,n_q)=(n,n-1)~\cite{LuceSeteLiesen2014b}.
\end{cases}
\end{align*}

Let $f$ be the \emph{Rhie function} from~\cite{Rhie2003} (see
also~\cite{SeteLuceLiesen2015b,SeteLuceLiesen2015a}), which has a rational function
$r$ of the type $(n-1,n)$, and which has $5n-5$ zeros. The results in~\cite{LiesenZur2017}
imply that for sufficiently small $|c| > 0$ the function $f_c$ has the same
number of zeros as $f$. The corresponding rational function $r_c$ then is of
the type $(n,n)$, so the bound $N(f)\le 5n-5$ is sharp also in the case $(n,n)$.
More generally, the sharpness of the bounds in Theorem~\ref{thm:main} 
for $n_p \ge n_q$ is discussed in~\cite[Theorem~C]{LeeMakarov2014}, while the case 
$n_p < n_q + 1$ remains a subject of future research.

\paragraph*{Acknowledgements}
We thank an anonymous referee for several helpful suggestions, and
in particular for pointing out the technical report~\cite{LeeMakarov2014}.

\bibliography{literature}
\bibliographystyle{siam}

\end{document}